\documentclass{elsarticle}
\usepackage[english]{babel}

\usepackage{hyperref}

\usepackage{graphicx}
\usepackage{pgfplots}
\pgfplotsset{compat=1.17}
\usepackage{subfig}

\usepackage{amsmath}
\usepackage{amssymb}
\usepackage{amsthm}
\usepackage{mathtools}
\newtheorem{theorem}{Theorem}[section]
\newtheorem{corollary}[theorem]{Corollary}
\newtheorem{lemma}[theorem]{Lemma}
\newdefinition{definition}[theorem]{Definition}
\newdefinition{remark}[theorem]{Remark}
\newdefinition{example}{Example}[section]
\usepackage{cleveref}

\newcommand{\opK}{\mathop{\vphantom{\sum}\mathchoice
  {\vcenter{\hbox{\huge K}}}
  {\vcenter{\hbox{\Large K}}}{\mathrm{K}}{\mathrm{K}}}\displaylimits}
\newcommand{\contf}[4]{
	\if\relax\detokenize{#2}\relax %if no upper bound is given, use \infty
		\opK_{#1}^{\infty}\left(\frac{#3}{#4}\right)
	\else
		\opK_{#1}^{#2}\left(\frac{#3}{#4}\right)
	\fi
}

\newcommand*{\bigO}{\mathcal{O}\opfences}

\newcommand*{\complns}{\mathbb{C}} %komplexe Zahlen
\DeclarePairedDelimiter{\ceil}{\lceil}{\rceil}
\DeclarePairedDelimiter{\floor}{\lfloor}{\rfloor}
\DeclarePairedDelimiter{\abs}{\lvert}{\rvert}

\DeclarePairedDelimiter{\opfences}{(}{)}
\newcommand*{\diag}{\operatorname{diag}\opfences}
\newcommand*{\transpose}[1]{{#1}^{\mathsf{T}}}
\newcommand*{\T}{\mathsf{T}}
\newcommand*{\vectorize}{\operatorname{vec}\opfences}

\newcommand*{\Tfive}{T_n}
\newcommand{\Tpert}{\mathfrak{T}} %perturbed T
\newcommand{\spec}{\operatorname{spec}}

\newcommand{\comment}[1]{} %textteile auskommentieren

\definecolor{myBlue}{HTML}{3333EE}
\definecolor{ManuelC}{HTML}{990000}

\date{\today}

\begin{document}

\begin{frontmatter} 
\title{Matrix functions via linear systems built from continued fractions}

\author[1]{Andreas Frommer}\ead{frommer@uni-wuppertal.de}\corref{cor}
\author[1]{Karsten Kahl}\ead{kkahl@uni-wuppertal.de}
\author[1]{Manuel Tsolakis}\ead{tsolakis@uni-wuppertal.de}
\address[1]{Department of Mathematics, Bergische Universit\"at Wuppertal, 42097 Wuppertal, Germany}
\cortext[cor]{Corresponding author}

\begin{abstract}
A widely used approach to compute the action $f(A)v$ of a matrix function $f(A)$ on a vector $v$ is to use a rational approximation $r$ for $f$ and compute $r(A)v$ instead. If $r$ is not computed adaptively as in rational Krylov methods, this is usually done using the partial fraction expansion of $r$ and solving linear systems with matrices $A- \tau I$ for the various poles $\tau$ of $r$. Here we investigate an alternative approach for the case that a continued fraction representation for the rational function is known rather than a partial fraction expansion. This is typically the case, for example, for Pad\'e approximations. From the continued fraction, we first construct a matrix pencil from which we then obtain what we call the CF-matrix (continued fraction matrix), a block tridiagonal matrix whose blocks consist of polynomials of $A$ with degree bounded by 1 for many continued fractions. 
%The number of blocks corresponds to the length of the continued fraction and thus usually to the degree of the rational function. 
We show that one can evaluate $r(A)v$ by solving a single linear system with the CF-matrix and 
present a number of first theoretical results as a basis for an analysis of future, specific solution methods for the large linear system. 
%which allow to investigate in some detail the connection between splitting-based iterative solvers for the CF-matrix and corresponding ones for the shifted linear systems arising from the partial fraction expansion. 
While the CF-matrix approach is of principal interest on its own as a new way to compute $f(A)v$, 
it can in particular be beneficial when a partial fraction expansion is not known beforehand and computing its parameters is ill-conditioned. 
We report some numerical experiments which show that with standard preconditioners we can achieve fast convergence in the iterative solution of the large linear system. 
\end{abstract}

\end{frontmatter}

\section{Introduction}
If $A\in \complns^{m \times m}$ and $f: D \subset \complns \to \complns$, the matrix function $f(A)$ is defined provided $f$ is $n(\lambda)-1$ times differentiable at all eigenvalues $\lambda$ of $A$ with $n(\lambda)$ being their multiplicity in the minimal polynomial of $A$. For diagonalizable $A$,
i.e., $A = W\Lambda W^{-1}$ with $\Lambda = \diag{\lambda_1,\ldots,\lambda_n}$ (implying $n(\lambda_i) =1$ for all $i$), one has 
\[
f(A) = W \diag{f(\lambda_1),\ldots,f(\lambda_n)} W^{-1},
\]
and if $A$ is not diagonalizable, a similar representation for $f(A)$ can be given involving the Jordan blocks of $A$. We refer the reader to  \cite{Higham}  for the above facts as well as for many further results on properties and computational methods for matrix functions. 

Even when $A$ is sparse, $f(A)$ is typically a full matrix. This implies that if $A$ is large and sparse, we are practically bound to compute the action $f(A)v$ of $f(A)$ on a vector $v$ rather than computing the full matrix $f(A)$, and it is this situation that we consider in this paper. 

Most of the existing methods fall into one of the two following categories: 
$f(A)v$ is computed approximately as the action of a matrix polynomial $p(A)$ on $v$, or it is approximated as the action of a rational matrix function $r(A)$ on $v$. The polynomial $p$ or the rational function $r$ can either be obtained by using \emph{a priori} information about the matrix $A$ like information about its spectrum or its numerical range, or it can be constructed adaptively in a process that accumulates increasingly accurate information on the spectral properties of $A$, depending on the vector $v$.  

For example, if $A$ is Hermitian and if we know an interval $[a,b]$ containing 
the spectrum  of $A$, we can use Chebyshev best polynomial or rational 
approximations for $f$ on $[a,b]$ or approximations to these best approximations 
like truncated Chebyshev series approximations, see, e.g.,  \cite{Higham,Gallopoulos,Kenney,Moler,Trefethen}. On the other hand, Arnoldi approximations for $f(A)v$ take their values
from the Krylov subspace spanned by $A$ and $v$ and as such adaptively produce  
polynomial approximations, depending not only on $A$ but also on $v$. The same 
holds true for the rational approximations obtained in rational Krylov subspace 
methods. We refer to \cite{FrSi06} or \cite{Higham} for an 
overview of Arnoldi-type methods and to the review \cite{Guettel} for 
rational Krylov methods. 

In polynomial methods, the computational cost is determined by matrix-vector 
products with $A$, whereas in rational approximation methods the cost typically
resides in solving linear systems with matrices $A - \tau I$, where $\tau$ is a 
pole. Of course, solving linear systems is usually much more expensive than a 
matrix-vector multiplication, but the additional cost when using rational 
approximations is often more than balanced for by the fact that we get much better approximation properties. 

In this paper, we study the situation where a rational approximation $r$ to $f$ is given via a continued fraction expansion. This is naturally the case, for example, for Pad\'e approximations. It is well known that one can evaluate a continued fraction 
solving a linear system with a tridiagonal matrix, see, e.g., \cite{Haydock75,Ozaki07}. In a similar fashion, 
we use in this paper the continued fraction (``CF'') to construct a block tridiagonal ``CF-matrix'' of size $m(n+1)$, where $n$ is the degree of the
continued fraction, such that we can retrieve $r(A)v$ as the first $m$ 
components of a linear system involving the CF-matrix. We then investigate 
spectral properties of the CF-matrix. In this manner, we contribute to establishing and analyzing a new way 
to approximate $f(A)v$ using a rational function for which we do not have to know or compute its poles. %This can be advantageous when the poles are ill-conditioned.
%The price to pay is that we have to solve one block tridiagonal
%linear system with a matrix that has $n+1$ times the size of $A$. 
Our investigations are meant to provide first results which lay the ground for being able to develop particularly efficient solution methods for the CF-matrix in the future like, for example, adequately preconditioned Krylov subspace methods or multigrid approaches. The numerical examples in this paper illustrate first steps into this direction.

The paper is organized as follows: We recall the most important definitions and properties of continuous fractions in \cref{sec:continued_fractions}.
In \cref{sec:5d_matrix} we then develop in detail how to construct the CF-matrix from which we get $f(A)v$ as a part of the solution for a particular right-hand side. \Cref{sec:simple} discusses special but important cases in which the CF-matrix takes a particularly appealing form in the sense that no higher powers of $A$ appear in its blocks.
This allows us to fully characterize the partial fraction expansion form of the rational function via the Weierstrass canonical form of certain matrix pencils and to establish relations between
a block Gauss-Seidel and a block Jacobi iteration on the CF-matrix with Gauss-Seidel and Jacobi performed on the systems $A-\tau_iI$, $\tau_i$ being the poles. \Cref{sec:numerics} presents a few numerical examples comparing the convergence of preconditioned GMRES \cite{Saad_GMRES} on the CF-matrix and on the systems resulting from the partial fraction expansion.

\section{Continued fractions and rational functions} \label{sec:continued_fractions}
We start by reviewing some properties of continued fractions that can be found in standard literature, e.g., \cite{Cuyt}. 
 
Given a formal continued fraction as
\begin{equation*}
	g = b_0 + \contf{i=1}{}{c_i}{b_i} = b_0 + \cfrac{c_1}{b_1 + \cfrac{c_2}{b_2 + \dots}}, \qquad b_i, c_i \in \complns
\end{equation*}
its $n$-th \emph{approximant} (sometimes also called \emph{convergent}) is defined by
\begin{equation*} \label{eq:approx:def}
	g_n = b_0 + \contf{i=1}{n}{c_i}{b_i} = b_0 + \cfrac{c_1}{b_1 + \cfrac{c_2}{b_2 + \quad\raisebox{-3mm}{$\ddots$}\quad\raisebox{-6mm}{$+\cfrac{c_n}{b_n}$}}}
\end{equation*} with corresponding \emph{tail}
\begin{equation*} \label{eq:tail_def}
    t_{n} = \contf{i=n+1}{\infty}{c_i}{b_i}.
\end{equation*}
In this definition, we always assume that either $c_i \neq 0$ for all $i$ or that if $c_n = 0$ for some $n$, then $b_i = 1$ and $c_i = 0$ for all $i \geq n$. In the latter case, we say that the continued fraction is 
{\em finite} and also write \[
g = b_0 + \contf{i=1}{n}{c_i}{b_i} = g_n.
\]
With these assumptions, considering the usual extension of complex 
arithmetic to $\complns \cup \{\infty\}$, we see that each approximant $g_n$ to a formal continued fraction is defined with 
value in  $\complns \cup \{\infty\}$, and we say that $g$ exists with value $\lim_{n \to \infty} g_n \in \complns \cup \{\infty\}$ if this limit exists.

The approximant $g_n$ of a continued fraction can be expressed as a simple fraction $g_n = \frac{p_n}{q_n}$ with $p_n$, $q_n$ determined by the recursion
\begin{equation}
\label{eq:cf_recursion}
%\begin{cases}\quad
	%\begin{array}{l}
		\begin{bmatrix} p_{-1} \\ q_{-1} \end{bmatrix} = \begin{bmatrix} 1 \\ 0 \end{bmatrix},\quad  
		\begin{bmatrix} p_{0} \\ q_{0} \end{bmatrix} = \begin{bmatrix} b_0 \\ 1 \end{bmatrix}, \quad  %\\[2em]%  \enspace
		\begin{bmatrix} p_n \\ q_n \end{bmatrix} = b_n \begin{bmatrix} p_{n-1} \\ q_{n-1} \end{bmatrix} + c_n \begin{bmatrix} p_{n-2} \\ q_{n-2} \end{bmatrix} \enspace \text{for } n \geq 1. 
	%\end{array}
%	\end{cases}
\end{equation}
The tails obey the simple ``backward'' recursion
\begin{equation} \label{eq:tail_rec}
    t_n = \frac{c_{n+1}}{b_{n+1}+t_{n+1}} \enspace \text{\ for\ } n=0,1,\ldots,
\end{equation}
with $b_0 + t_0 = g$.

\begin{remark} \label{rem:conf_rec_zero}
The recursion  \cref{eq:cf_recursion} allows that for a given $n$ either $p_n$ or $q_n$ can be zero (which then gives $g_n = 0$ or 
$g_n = \infty$), but that $p_n$ and $q_n$ cannot vanish at the same time. Indeed, if this were the case, the recursion shows that the vectors $\begin{bsmallmatrix} p_i\\ q_i\end{bsmallmatrix}$ end up being all collinear for $i=n-1, n-2, \ldots, -1$, an obvious contradiction to the definition of these vectors for $i=0$ and $i=-1$. 
\end{remark}

Note that the value of a continued fraction does not change if we expand each inner fraction with arbitrary factors $d_i\neq 0$ in the sense that
\begin{equation}
\label{eq:expanding_cf}
	g = b_0 + \contf{i=1}{}{c_i}{b_i} = b_0 + \contf{i=1}{}{d_{i-1}d_i c_i}{d_i b_i}, \quad d_0 = 1.
\end{equation}

Continued fractions can be extended to also represent functions of a (complex) variable $z$. This is usually achieved by using polynomials as partial numerators and denominators,
\begin{equation*}
	g(z) = b_0(z) + \contf{i=1}{}{c_i(z)}{b_i(z)} \text{\ with\ } c_i(z), b_i(z) \text{\ polynomials in\ } z.
\end{equation*}
The approximants $g_n(z)$ are then rational functions, see \cref{eq:cf_recursion}. For instance, continued fractions of the form
\begin{equation}
\label{eq:C_frac}
	g(z) = b_0 + \contf{i=1}{\infty}{c_i z}{1}
\end{equation}
are called \emph{regular C-fractions}.
It is known that for every Stieltjes function $f(z)$ (see \cite{Henrici} for a definition), there exists a regular C-fraction such that the approximants yield a \emph{descending staircase} of Padé approximations for $f(z)$, i.e.,
\begin{equation*}
	g_n(z) = r_{\ceil{\frac{n}{2}}, \floor{\frac{n}{2}}}(z).
\end{equation*}

Recall that the Padé approximation $r_{k,l}(z) = \frac{p_{k,l}(z)}{q_{k,l}(z)}$ is the rational function with numerator degree $k$ and denominator degree $l$ for which at least the first $k+l$ terms of the Taylor expansion at 0 agree with the approximated function $f(z)$,
\begin{equation*}
	f(z) - \frac{p_{k,l}(z)}{q_{k,l}(z)} = \bigO{z^{k+l+1}}.
\end{equation*}
 We refer to \cite{Baker} as a reference for the theory of Pad\'e approximations and to \cite{Baker, Cuyt} for the many connections 
 which exist between Pad\'e approximations and continued fractions.

C-fractions can be \emph{contracted}, see \cite{Cuyt}, which gives the new continued fraction
\begin{equation}
\label{eq:C_frac_contracted}
	\begin{gathered}
		\widetilde{g}(z) = b_0 + \frac{c_1 z}{1 + c_2 z + \contf{i=2}{\infty}{-c_{2i-2}c_{2i-1} z^2}{1 + (c_{2i}+c_{2i-1})z}},\\
		\widetilde{g}_{n}(z) = r_{n,n}(z) = g_{2n}(z),
	\end{gathered}
\end{equation}
in which the $n$-th approximant corresponds to the $2n$-th approximant of the original C-fraction. We can thus represent the diagonal $(n,n)$-Pad\'e approximation of a Stieltjes function as a contracted C-fraction of the form \cref{eq:C_frac_contracted}.
% and thus yields the diagonal $(n,n)$-Padé approximation.

More generally, for each $j\in \mathbb{N}$ the Padé approximations $r_{j+\ceil{\frac{n}{2}}, \floor{\frac{n}{2}}}(z)$ can be obtained as the approximants of a continued fraction of the form
\begin{equation*}
	r_{j,0}(z) + z^j\contf{i=1}{\infty}{c_i z}{1}.
\end{equation*}
For Padé approximations with $k<l$, we note that the reciprocal rational function $r_{k,l}(z)^{-1} = \frac{q_{k,l}(z)}{p_{k,l}(z)}$ is the Padé approximation of $f(z)^{-1}$ with switched degrees. As the inverse of a continued fraction is again a continued fraction
\begin{equation*}
	g = b_0 + \contf{i=1}{}{c_i}{b_i} \quad\implies\quad g^{-1} = 0 + \contf{i=0}{}{c_i}{b_i}, \quad c_0=1,
\end{equation*}
one can construct Padé approximants with $k<l$ by considering $f(z)^{-1}$.

Every rational function can be written as a finite continued fraction. One way to obtain such a finite continued fraction is by using the Euclidean algorithm with polynomial long division to find the greatest common divisor of the polynomials $p(z)$ and $q(z)$, see, e.g., \cite{Knopfmacher}, 
\begin{equation*}
	\frac{p(z)}{q(z)} = b_0(z) + \frac{a(z)}{q(z)} = b_0(z) + \frac{1}{\frac{q(z)}{a(z)}}.
\end{equation*}
Here, the degree of the remainder $a(z)$ is less than the degree of $q(z)$, and one recurses with the new rational function $\frac{q(z)}{a(z)}$ until the remainder is  zero. We write the resulting continued fraction as
\begin{equation*}
	\frac{p(z)}{q(z)} = b_0(z) + \contf{i=1}{n}{1}{b_i(z)}.
\end{equation*}

Note that the continued fractions obtained in this manner are not C-fractions: 
The polynomials are in the partial denominators instead of the partial numerators. In particular, even if one starts at a Padé approximation, the process just described yields a continued fraction whose approximants are not necessarily Padé approximations of lower degree. 

\section{CF-matrices} \label{sec:5d_matrix}
In this section, we show how the inverse of the approximant $g_n$ of any (formal) continued fraction $g$ can be obtained as the $(1,1)$ entry of the inverse of a tridiagonal matrix $T_n$ built from the partial numerators and denominators. By extension, we show the connection between the action $r(A)v$ of a rational matrix function $r(A)$ on a vector $v$ and the solution of a single, block tridiagonal linear system.

\subsection{Continued fractions and tridiagonal matrices}
\begin{theorem} \label{thm:trid}
    Let the continued fraction $g = b_0 + \contf{i=1}{}{c_i}{b_i}$ be given and $g_n$ denote its $n$-th approximant with $g_n \neq 0$. If the entries of the tridiagonal matrix
    \begin{equation*}
    	T_n  =	\begin{bmatrix}
    				\beta_{0} & \gamma_{1} \\
    				\alpha_{1} & \beta_{1} & \ddots \\ 
    						& \ddots  & \ddots 		& \ddots \\
    						&		  & \ddots & \beta_{n-1}	& \gamma_{n} \\
    						&		  & 			& \alpha_{n}		& \beta_{n}
    			\end{bmatrix}
    			\in \complns^{(n+1)\times(n+1)}
    \end{equation*}
    fulfill
    \begin{align*}
    	\beta_{i} = b_i,\quad i=0,\ldots,n, \\
    	-\alpha_i \gamma_i = c_i,\quad i=1,\ldots,n,
    \end{align*}
    then the matrix $T_n$ is nonsingular and
    \begin{equation*}
    	g_n^{-1} = (T_n^{-1})_{1,1}.
    \end{equation*}
\end{theorem}
\begin{proof}
    $(T_n^{-1})_{1,1}$ is the first entry of the solution $x$ of the linear system $T_n x = e_1$ with $e_1$ being the first unit vector. Using Cramer's rule we have
    \begin{equation} \label{Cramer:eq}
    	(T_n^{-1})_{1,1} = \transpose{e}_1 T_n^{-1} e_1 = \frac{\transpose{e}_1\operatorname{adj}(T_n)e_1}{\det(T_n)} = \frac{\det(T_n')}{\det(T_n)},
    \end{equation}
    where $T_n'$ is obtained from $T_n$ by replacing $\beta_0$ with $1$ and $\alpha_1$ with 0.
    Let $d_i = \det(T_i)$ denote the $i$-th principal minor of $T_n$. Then we have the recurrence relation
    \[
    %\begin{cases}
    %\quad 
     d_{-1} = 1,\quad d_0 = \beta_0, \quad
    d_i = \beta_i d_{i-1} - \alpha_i \gamma_i d_{i-2}, \quad \mbox{for }  i=1, \dots, n.
    %\end{cases}
    \]
    Due to $\beta_i = b_i$ and $\alpha_i\gamma_i = -c_i$, this is exactly the recursion for the numerators $p_i$ from \cref{eq:cf_recursion}, i.e., $d_i = p_i$. As we assumed $g_n \neq 0$, we have $d_n = p_n \neq 0$ and thus $T_n$ is nonsingular. In a similar manner, we obtain that the principal minors $d'_i$ of $T'$ satisfy the same recursion as $q_i$ from  \cref{eq:cf_recursion}. Thus, for $i=n$, we obtain from  \cref{Cramer:eq} that
    \[
    (T_n^{-1})_{1,1} = \frac{q_n}{p_n} = g_n^{-1}. \qedhere
    \]
\end{proof}
\begin{remark} If $T_n$ is singular, we still have $p_n = d_n$, but now with value $0$. \cref{rem:conf_rec_zero} shows that in this case $q_n = \det(T_n') \neq 0$. Thus \cref{thm:trid} also holds in the case $g_n = 0$, where $T_n$ is singular, if we interpret  $(T_n^{-1})_{1,1}$ as $\infty$. \end{remark}

\begin{corollary} \label{cor:freedom}
	\Cref{thm:trid} still holds if $T_n$ is multiplied with any two nonsingular matrices from the left and the right as long as $e_1$ is a right and left eigenvector, respectively, of these matrices with reciprocal eigenvalue:
	\begin{equation*}
		\left.\begin{aligned}
			H_\ell^{-1} e_1 &= \lambda e_1 \\
			\transpose{e}_1 H_r^{-1} &= \lambda^{-1} \transpose{e}_1
		\end{aligned}\right\}
		\implies g_n^{-1} = \transpose{e}_1 T_n^{-1} e_1 = \transpose{e}_1 (H_\ell T_n H_r)^{-1} e_1.
	\end{equation*}
	In particular, using diagonal matrices $D=\diag{1, d_1, \dots, d_n}$ for $H_\ell$ and $H_r$ is equivalent to expanding the continued fraction as in \cref{eq:expanding_cf}.
\end{corollary}

\subsection{Extension to matrix functions}
We now consider (formal) continued fractions as functions by assuming that the coefficients of the continued fraction $g(z)$ are polynomials in $z\in\complns$ of maximum degree $\ell$,
\begin{equation} \label{eq:poly_conf_def}
	g(z) = b_0(z) + \contf{i=1}{}{c_i(z)}{b_i(z)},\text{\ where\ } b_i(z) = \sum_{j=0}^\ell b_i^{(j)}z^j \text{\ and\ }
    c_i(z) = \sum_{j=0}^\ell c_i^{(j)}z^j. 
\end{equation}
Define the tridiagonal matrices
\[
T_n^{(j)} = \begin{bmatrix}
				\beta_{0}^{(j)} & \gamma_{1}^{(j)} \\
				\alpha_{1}^{(j)} & \beta_{1}^{(j)} & \ddots \\
						& \ddots  & \ddots 		& \ddots \\
						&		  & \ddots & \beta_{n-1}^{(j)}	& \gamma_{n}^{(j)} \\
						&		  & 			& \alpha_{n}^{(j)}		& \beta_{n}^{(j)}
				\end{bmatrix}, \enspace j=0,\ldots,\ell,
\]
and 
\begin{equation} \label{eq:5d_sum}
    T_n(z) = \sum_{j=0}^{\ell} T_n^{(j)} z^j = \begin{bmatrix}
				\beta_{0}(z) & \gamma_{1}(z) \\
				\alpha_{1}(z) & \beta_{1}(z) & \ddots \\
						& \ddots  & \ddots 		& \ddots \\
						&		  & \ddots & \beta_{n-1}(z)	& \gamma_{n}(z) \\
						&		  & 			& \alpha_{n}(z)		& \beta_{n}(z)
				\end{bmatrix},
\end{equation}
where 
%\[
$\alpha_i(z) = \sum_{j=0}^{\ell} \alpha_i^{(j)}z^j$, $\beta_i(z) = \sum_{j=0}^{\ell} \beta_i^{(j)}z^j$,
$\gamma_i(z) = \sum_{j=0}^{\ell} \gamma_i^{(j)}z^j$.
%\]
Then for all $z$ for which $c_i(z) \neq 0$ for $i=1,\ldots,n$, \cref{thm:trid} shows that we have $g_n(z)^{-1} = (T_n(z)^{-1})_{1,1}$ provided
\begin{equation} \label{eq:polynomial_identities}
\left\{
\begin{array}{rl}
\beta_i(z) = b_i(z), & i=0,\ldots,n,  \\
-\alpha_i(z)\gamma_i(z) = c_i(z), & i=1,\ldots,n.
\end{array}
\right. 
\end{equation}
By continuity, \cref{eq:polynomial_identities} also implies $g_n(z)^{-1} = (T_{n}(z)^{-1})_{1,1}$ for those $z$ which are a zero of one of the $c_i$.

Note that \cref{eq:polynomial_identities} means that $\beta_i^{(j)} = b_i^{(j)}$ for the coefficients of the polynomial $\beta_i$, but that we have freedom in choosing the $\alpha_i^{(j)}$ and $\gamma_i^{(j)}$ in the sense that only the product of the two polynomials $\alpha_i$ and $\gamma_j$ with these coefficients is prescribed. This means that for the coefficients we in general have $-\alpha_i^{(j)}\gamma_i^{(j)} \neq c_i^{(j)} $. 
\begin{remark}
    Consider the special case
    \begin{equation*}
        T_n(z) = T_n^{(0)} - zI.
    \end{equation*}
    From \cref{thm:trid}, we know that the continued fraction $g_n$ constructed from the entries of $T_n(z)$ fulfills
    \begin{equation*}
        g_n(z)^{-1} = \frac{q_n(z)}{p_n(z)} = 
        \frac{q_n(z)}{\det(T_n^{(0)} - zI)}.
    \end{equation*}
    Hence, the eigenvalues of a tridiagonal matrix are the zeros of the continued fraction constructed from its entries, a known fact that dates back at least to Rutishauer; see \cite[Anhang~§1]{Rutishauser}.
\end{remark}

Let us now consider a matrix $A\in\complns^{m \times m}$ instead of $z\in\complns$ and assume that the function $g(z)$, given by a continued fraction, is defined on the spectrum of $A$ in the sense of \cite{Higham}. % the book to which we refer for further reading about basic definitions and properties of matrix functions which we cannot reproduce here.
Any approximant $g_n(z)$ of $g(z)$ is a rational function of $z$ and as such defined on the spectrum of $A$ as long as no eigenvalue of $A$ is a pole of $g_n(z)$. 

The value of a matrix function is independent
from the way we represent the function, see \cite{Higham}, so that with a slight abuse of notation we can write 
\begin{equation*}
    g_n(A) = b_0(A) + \contf{i=1}{n}{c_i(A)}{b_i(A)},
\end{equation*}
where each denominator is to be understood as a matrix inversion. On the other hand, let us define the matrix $T_n(A)$ as 
\begin{equation} \label{eq:5d_sumA}
	T_n(A) = \begin{bmatrix}
	\beta_{0}(A) & \gamma_{1}(A) \\
				\alpha_{1}(A) & \beta_{1}(A) & \ddots \\
						& \ddots  & \ddots 		& \ddots \\
						&		  & \ddots & \beta_{n-1}(A)		& \gamma_{n}(A) \\
						&		  & 			& \alpha_{n}(A)	& \beta_{n}(A)
			 \end{bmatrix}
			 = \sum_{j=0}^{\ell} T_n^{(j)} \otimes A^j,
\end{equation}
where $\otimes$ denotes the Kronecker product. Note that $T_n(A)\in\complns^{(n+1)m \times (n+1)m}$ is a \emph{block} tridiagonal matrix. 
\begin{definition}
	The matrix $T_n(A)\in\complns^{(n+1)m \times (n+1)m}$ from \cref{eq:5d_sumA}, constructed from the approximant of a continued fraction with polynomial partial numerators and denominators, is called \emph{CF-matrix}. (CF stands for ``continued fraction''.)
\end{definition}

Due to its construction, we already know that $\beta_i(A) = b_i(A)$ and $-\alpha_i(A)\gamma_i(A)=c_i(A)$ and so one might wonder whether---by analogy with \cref{thm:trid}---the block $(T_n(A)^{-1})_{1,1}$ yields the matrix function $g_n(A)^{-1}$ of the continued fraction $g_n(z)^{-1}$. To show that this is indeed true, let us first introduce the UDL decomposition for block tridiagonal matrices.
\begin{lemma}
\label{lem:UDL}
	Let $T\in\complns^{(n+1)m \times (n+1)m}$ be a block tridiagonal matrix and denote its blocks by $T_{i,j} \in\complns^{m \times m}$. If $T$ is nonsingular and if all matrices $\Sigma_i$ defined below are nonsingular, too, the following decomposition exists:
	\[
		T = UDL, 
	\]
	where $D = \diag{\Sigma_0, \dots, \Sigma_n}$ and 
	\begin{equation*}
		U = \renewcommand*{\arraystretch}{2} \begin{bmatrix}  	   
				I & T_{0,1} \Sigma_1^{-1} \\
		%		  & I & T_{1,2} \Sigma_2^{-1} \\
				    & \ddots 	& \ddots \\
				     & 		& I 	& T_{n-1,n} \Sigma_n^{-1} \\
				     &			&		& I
			\end{bmatrix}, \enspace
		L =  \renewcommand*{\arraystretch}{2}  \begin{bmatrix}
				I \\ 
				\Sigma_1^{-1} T_{1,0} 	& I \\
														& \ddots 	& \ddots \\
														&			& \Sigma_n^{-1} T_{n,n-1} & I
			\end{bmatrix} .
	\end{equation*}
	Herein, the $\Sigma_i$ are the block Schur complements described by the backward recursion
	\begin{equation*} \label{eq:Schur_compl_rec}
	\Sigma_n = T_{n,n}, \enspace \Sigma_i = T_{i,i} - T_{i,i+1}\Sigma_{i+1}^{-1}T_{i+1,i}, \enspace i= n-1,\ldots,1.
	\end{equation*}
\end{lemma}
\begin{proof}
	Straightforward algebra.
\end{proof}

For a CF-matrix $T_n(A)$, the Schur complements---if they exist---are rational matrix functions of $A$. Since these commute with polynomials in $A$, we see that
\begin{equation*}
    \Sigma_i = \beta_i(A) - \gamma_{i+1}(A)\Sigma_{i+1}^{-1}\alpha_{i+1}(A) = b_i(A) + c_{i+1}(A)\Sigma_{i+1}^{-1}.
\end{equation*}
Having thus expressed the recursions for the Schur complements in terms of the $b_i$ and $c_i$, we see from \cref{eq:tail_rec} that we actually have
\[
\Sigma_i = b_i(A) + t_i(A),
\]
where $t_i(z)$ is the tail of the (finite) continued fraction $g_n$. In particular,  $\Sigma_0 = g_n(A)$.
The following theorem is now the matrix analog of \cref{thm:trid}.
\begin{theorem}
\label{thm:5d_linear}
	Let $g(z)$ be a continued fraction with polynomial partial numerators and denominators of the form \cref{eq:poly_conf_def} and let its $n$-th approximant $g_n(z)$ be the inverse of the rational function $r(z) = g_n(z)^{-1}$. Then
	\begin{equation} \label{eq:rational_func_TA}
		r(A)v = (\transpose{e}_1\otimes I) T_n(A)^{-1} (e_1\otimes v),
	\end{equation}
	i.e., the action of the rational function on a vector $v\in\complns^m$ can be computed by solving a linear system with the CF-matrix $T_n(A)$.
\end{theorem}
\begin{proof}
	We first assume that for the given matrix $A$ the UDL decomposition  $T_n(A)^{-1} = L^{-1}D^{-1}U^{-1}$ of \cref{lem:UDL} exists, i.e., that all Schur complements $\Sigma_i$ are nonsingular. Then, using $(e_1\otimes v) = (e_1\otimes I)v$,
	\begin{equation*}
		(\transpose{e}_1\otimes I) T_n(A)^{-1} (e_1\otimes v) = (\transpose{e}_1\otimes I) L^{-1}D^{-1}U^{-1} (e_1\otimes I) v.
	\end{equation*}
	Analogously to \cref{cor:freedom}, since $(\transpose{e}_1\otimes I) L^{-1} = \transpose{e}_1\otimes I$, $U^{-1} (e_1\otimes I) = e_1\otimes I$, we can simplify this to
	\begin{equation*}
		(\transpose{e}_1\otimes I) L^{-1} D^{-1} U^{-1} (e_1\otimes I) = %(\transpose{e}_1\otimes I) L^{-1} D^{-1} (e_1\otimes I) = 
		(\transpose{e}_1\otimes I) D^{-1} (e_1\otimes I) = \Sigma_0^{-1}.
	\end{equation*}
	By construction, $\Sigma_0 = g_n(A) = r(A)^{-1}$, thus giving \cref{eq:rational_func_TA}.
	
	For a general matrix $A$, the theorem now follows by a continuity argument: The Schur complements $\Sigma_i$ are
	rational matrix functions $s_i(A)$, with $s_i$ defined by the entries of $T_n(z)$, evaluated at the matrix $A$.  The 
	Schur complements are thus nonsingular if $s_i(\lambda) \neq 0$ for all $\lambda \in \spec(A)$ and 
	$i=0,\ldots,n$. Let $A = VJV^{-1}$ be the Jordan canonical form of $A$, take $D_\varepsilon = (1+\varepsilon) I$ with 
	$\varepsilon >0$ and consider $A_\varepsilon = VD_\varepsilon JV^{-1}$. Then $\spec(A_\varepsilon) = \{(1+\epsilon)\lambda,\; \lambda \in \spec(A)\}$. The set of all zeros of all the 
	rational functions $s_i(z)$ is finite and so is $\spec(A)$. This implies that for $\varepsilon>0$ sufficiently small the 
	set $\spec(A_\varepsilon)$ is disjoint from the set of all zeros of all $s_i$. For such $\varepsilon$, from what we have 
	already shown, we know that \cref{eq:rational_func_TA} holds for $A_\varepsilon$, and letting $\varepsilon \to 0$ shows that \cref{eq:rational_func_TA} also holds for $A$.
\end{proof}
\begin{remark}
	If the approximants $g_n$ of $g(z)$ yield Padé approximations for each $n$, we can increase the degree of the Pad\'e approximation by simply appending $k$ block rows and columns to $T_n(A)$ which yields $T_{n+k}(A)$.
\end{remark}

\subsection{Eigendecomposition}
Assume that $A \in \complns^{m \times m}$ is diagonalizable with the eigendecomposition
\begin{equation*}
	AW = W\Lambda,
\end{equation*}
where the columns $w_i$ of $W$ are the eigenvectors and the diagonal elements $\lambda_{i}$ of the diagonal matrix $\Lambda$ the corresponding eigenvalues.
Then it is easy to find the eigendecomposition of the general CF-matrix $T_n(A)$: For any vector $v \in \complns^{n+1}$ we have
\begin{align*}
    \MoveEqLeft T_n(A) (v \otimes w_i) \\
        &= \sum_{j=0}^{\ell} (T_n^{(j)} \otimes A^j)(v \otimes w_i) = \sum_{j=0}^{\ell}   T_n^{(j)}  v \otimes A^j w_i = \left(\sum_{j=0}^{\ell} \lambda_i^j T_n^{(j)} v \right) \otimes w_i \\
        &= \left( T_n(\lambda_i) v \right) \otimes w_i.
\end{align*}
Thus, for each $i$, if $v$ is an eigenvector of $T_n(\lambda_i)$, then $v \otimes w_i$ is an eigenvector of $T_n(A)$. As a consequence, if each of the matrices $T_n(\lambda_i)$ is diagonalizable with eigenvectors $v_{k,i}$  and eigenvalues $\mu_{k,i}$, then the $(n+1)m$ vectors $v_{k,i} \otimes w_i$ represent a full system of eigenvectors for $T_n(A)$ with eigenvalues $\mu_{k,i}$.

\section{Special CF-matrices}\label{sec:simple}
Up until now, the polynomials that appear in the continued fraction were allowed to be of arbitrarily high degree, and thus, expressing $T_n(A)$ as in \cref{eq:5d_sum}, might involve many terms. In many cases, however, we only need the first two terms. %In addition, we will drop the index $n$ for the CF-matrix and its components.
In this section, we now assume that $\Tfive(z) = \Tfive^{(0)} - z\Tfive^{(1)}$, such that
\begin{equation*}
	\Tfive(A) = \Tfive^{(0)}\otimes I - \Tfive^{(1)}\otimes A
\end{equation*}
with tridiagonal matrices $\Tfive^{(0)}, \Tfive^{(1)} \in \complns^{(n+1)\times(n+1)}$. We take a closer look at the linear system 
\begin{equation}
\label{eq:5d_linearsystem}
	\Tfive(A)x = e_1 \otimes v,
\end{equation}
the solution of which gives $g_n(A)^{-1}v$ according to \cref{thm:5d_linear}.

%Since we are interested in $r(A)v$, which is equal to $(\transpose{e}_1\otimes I) \Tfive(A)^{-1} (e_1\otimes v)$ by \cref{thm:5d_linear}, we also consider the linear system
%\begin{equation}
%\label{eq:5d_linearsystem}
%	\Tfive(A)x = e_1 \otimes v.
%\end{equation}

\subsection{Construction}
To illustrate the connection between the matrices $\Tfive^{(0)}$, $\Tfive^{(1)}$ and the partial numerators and denominators of the underlying continued fraction, we discuss three special cases.
%Remember that we can construct $\Tfive^{(0)}$ and $\Tfive^{(1)}$ while considering $z \in \complns$ instead of $A$.
\begin{example}[Regular C-fractions]
	In the approximant of a regular C-fraction
	\begin{equation*}
		g_n(z) = b_0 + \contf{i=1}{n}{c_iz}{1},
	\end{equation*}
	the partial denominators are all 1. For the numerators we use $-c_iz = (-1)\cdot(c_iz)$ and obtain
	\begin{equation*}
		\Tfive^{(0)} = \begin{bmatrix}
				b_0 & 1 \\
					& 1	& \ddots \\
					&	&\ddots & 1 \\
					&	&	& 1
			\end{bmatrix}, \quad
		\Tfive^{(1)} = \begin{bmatrix}
				0 	&  \\
				c_1 & \ddots&  \\
					& \ddots&\ddots & \\
					&		& c_n 	& 0
			\end{bmatrix}.
	\end{equation*}
	Here, we chose the subdiagonal to contain the coefficients $c_i$, but we could as well choose the superdiagonal, i.e., we could take the pair $(\Tfive^{(0)})^{\T}$, $(\Tfive^{(1)})^{\mathsf{T}}$ instead of $\Tfive^{(0)}$, $\Tfive^{(1)}$.
\end{example}
\begin{example}[Contracted regular C-fractions]\label{ex:contracted}
	Assume we contract the regular C-fraction before constructing the CF-matrix. Recall that by  \cref{eq:C_frac_contracted} the approximants of the contracted partial fraction are given by 
	\begin{equation*}
		\widetilde{g}_{n}(z) = b_0 + \frac{c_1 z}{1 + c_2 z + \contf{i=2}{n}{-c_{2i-2}c_{2i-1} z^2}{1 + (c_{2i}+c_{2i-1})z}}.
	\end{equation*}
	In this continued fraction, monomials of degree 2 appear in the partial numerators. 
	Factorizing $c_{2i-2}c_{2i-1} z^2 = (c_{2i-2}z)(c_{2i-1} z)$, we see that a CF-matrix can be constructed as $\Tfive(z) = \Tfive^{(0)} - z \Tfive^{(1)} $ with 
	\begin{align*}
		\Tfive^{(0)} &= \begin{bmatrix}
				b_0 & -1 \\
					& 1	&\\
					&		& \ddots	& \\
					&		& 			& 1 
			\end{bmatrix},\\
		\Tfive^{(1)} &= (-1)\cdot\begin{bmatrix}
				0 	&  0\\
				c_1 & c_2 	& c_2 \\
					& c_3 	& c_3+c_4	& c_4 \\
					&		& c_5 		& c_5+c_6	& \ddots \\
					&		&			& \ddots 	& \ddots & c_{2n-2} \\
					&		&			&			& c_{2n-1}& c_{2n-1}+c_{2n}
			\end{bmatrix}.
	\end{align*}
	Since the approximant $\widetilde{g}_{n}(z)$ of the contracted C-fraction is the same rational function as $g_{2n}(z)$ of the original C-fraction,
	using the contracted form for the same diagonal Padé approximation reduces the size of the CF-matrix from $2n+1$ to $n+1$.
\end{example}
\begin{example}[Continued fractions via polynomial long division]
	Let us consider an approximant of a continued fraction obtained by repeated polynomial long division
	\begin{equation*}
		g_n(z) = b_0(z) + \contf{i=1}{n}{1}{b_i(z)},
	\end{equation*}
	and assume that all quotients have degree at most 1, i.e., $b_i(z) = b_i^{(0)} - b_i^{(1)}z$. Then a possible construction for $\Tfive^{(0)}, \Tfive^{(1)}$ is
	\begin{equation*}
		\Tfive^{(0)} = \begin{bmatrix}
				b_0^{(0)} 	& -1 \\
						1	& \ddots&\ddots \\
							&\ddots	&\ddots & -1 \\
							&		& 1	& b_n^{(0)}
			\end{bmatrix}, \quad
		\Tfive^{(1)} = \begin{bmatrix}
				b_0^{(1)} 	&  \\
					& \ddots&  \\
					&		& b_n^{(1)}
			\end{bmatrix}.
	\end{equation*}
	In this case, it is possible to transform $\Tfive^{(0)}$ and $\Tfive^{(1)}$ along the lines of
	\cref{cor:freedom} such that they are symmetric. For instance, with the unitary diagonal matrices
	\begin{align*}
		D_\mathrm{L} &= \diag{(-1)^{\floor{0/2}}, \dots, (-1)^{\floor{n/2}}},\\
		D_\mathrm{R} &= \diag{(-1)^{\floor{(0+1)/2}}, \dots, (-1)^{\floor{(n+1)/2}}}
	\end{align*}
	we obtain
	\begin{align*}
		D_\mathrm{L} \Tfive^{(0)} D_\mathrm{R} &= \begin{bmatrix}
				(-1)^0 b_0^{(0)}& 1 \\
						1		&(-1)^1 b_1^{(0)}	&\ddots \\
								&\ddots				&\ddots	& 1 \\
								&					& 1		& (-1)^n b_n^{(0)}
			\end{bmatrix},\\
		D_\mathrm{L} \Tfive^{(1)} D_\mathrm{R} &= \diag{(-1)^0 b_0^{(1)}, \dots, (-1)^n b_n^{(1)}}
	\end{align*}
	or, alternatively,
	\begin{align*}
		D_\mathrm{R} \Tfive^{(0)} D_\mathrm{L} &= \begin{bmatrix}
				(-1)^0 b_0^{(0)}& -1 \\
						-1		&(-1)^1 b_1^{(0)}	&\ddots \\
								&\ddots				&\ddots	& -1 \\
								&					& -1		& (-1)^n b_n^{(0)}
			\end{bmatrix},\\
		D_\mathrm{R} \Tfive^{(1)} D_\mathrm{L} & = D_\mathrm{L} \Tfive^{(1)} D_\mathrm{R}.
	\end{align*}
\end{example}

Let us remark that these examples illustrate that the matrix $\Tfive^{(1)}$ is often singular.

%For these cases, we remark the following 
%\begin{corollary}
%	Let $A, B \in \complns^{m\times m}$ be two square matrices. Let $B$ be singular and assume that $A-zB$ is nonsingular for $z$ sufficiently large. Then the condition number of the matrix pencil $A-zB$
%	increases asymptotically at least as fast as $z$,
%	\begin{equation*}
%		\kappa_2(A-zB) = \norm{A-zB}_2 \norm{(A-zB)^{-1}}_2 = \Omega(\abs{z}).
%	\end{equation*}
%\end{corollary}
%\begin{proof}
%	On the one hand, we have
%	\begin{equation*}
%		\norm{A-zB}_2 \geq \abs{z}\norm{B}_2 - \norm{A}_2.
%	\end{equation*}
%	On the other hand, with $y$ being a normalized vector from the kernel of $B$, we have
%	\[
%	    \norm{(A-zB)^{-1}}_2^{-1} = \min_{\norm{x}_2=1} \norm{(A-zB)x}_2 
%	    \leq \norm{(A-zB)y}_2  = \norm{Ay}_2 \leq \norm{A}_2.
%	\]
%\end{proof}

\subsection{Weierstrass canonical form and partial fraction expansion}
If we have $\det(T_n(z)) \not\equiv 0$, i.e., the determinant does not vanish identically as a function of $z$, then the \emph{Weierstrass canonical form}\footnote{The Weierstrass canonical form is a special case of the \emph{Kronecker canonical form} for regular pencils.} exists, see \cite{Gantmacher}.
That is, there exist nonsingular matrices $U,V \in \complns^{(n+1) \times (n+1)}$ such that 
\begin{equation*}
    U(T_n^{(0)} - zT_n^{(1)})V = 
     \begin{bmatrix}
         J^{(0)} - zI_{n^{(0)}} \\
         & I_{n^{(1)}} - zJ^{(1)}
     \end{bmatrix},
\end{equation*}
with the Jordan matrices
\begin{equation} \label{eq:Kronecker_form}
\left\{ \begin{array}{rcl}
    J^{(0)} &=& \bigoplus_{j=1}^{k_0} J(\tau_j,n_j^{(0)}),  \enspace \tau_j \in \complns, \\
    J^{(1)} &=& \bigoplus_{j=1}^{k_1} J(0,n_j^{(1)}),
\end{array} \right.
\end{equation}
where $n^{(i)} = \sum_{j=1}^{k_i} n_j^{(i)}$, $n^{(0)}+n^{(1)} = n+1$
and $J(\mu,m)$ denotes a Jordan block of size $m$ given as
\[
J(\mu,m) = \mu I_m + S_{m}, \enspace S_{m} = \begin{bmatrix} 0 & 1  & 0 & \cdots &0 \\
                            \vdots  & \ddots & \ddots & \ddots & 0 \\
                            \vdots&   & \ddots   &  \ddots & 0  \\
                            \vdots &    &     & 0 & 1 \\
                            0 & \cdots & \cdots & 0   & 0
            \end{bmatrix}, \enspace I_{m}, S_{m}  \in \complns^{m \times m}.
\]
The $\tau_i$ are the (generalized) eigenvalues of the pencil $T_n^{(0)}-zT_n^{(1)}$. With the Weierstrass canonical form, we are able to describe the partial fraction expansion of $g_n(z)^{-1}$ via the pencil $T_n^{(0)}-zT_n^{(1)}$ as we show in the following theorem. In particular, the eigenvalues of $T_n^{(0)}-zT_n^{(1)}$ are the poles of $g_n(z)^{-1}$.
%\af{Under the assumption that $T_n^{(1)}$ is nonsingular and that all poles are simple, \cite{Berljafa} develops an algorithm which computes the partial fraction expansion of the rational function $\transpose{e}_1T_n(z)^{-1} e_1$ from the pencil $T_n^{(0)}-zT_n^{(1)}$. We add to this from the theoretical side by showing how even without these additional assumptions on $T_n^{(1)}$ and the multiplicities of the poles, the complete partial fraction expansion can be obtained via the Weierstrass canonical form.} 
%
\begin{theorem} \label{thm:5d_pfe}
Assume that $\det(T_n(z)) \not\equiv 0$. Let $U,V$ and $J(\tau_j,n_j^{(0)})$, $J(0, n_j^{(1)})$
be the matrices and parameters of the Weierstrass canonical form \cref{eq:Kronecker_form}. Let $u = Ue_1$ and $\transpose{v}= \transpose{e}_1V$, and let $u^{(j)}$ and $(v^{(j)})^\T$ denote the blocks of $u$ and $\transpose{v}$ corresponding to block $j$ of the Weierstrass canonical form. 
Then
\begin{equation*}
    \transpose{e}_1 T_n(z)^{-1} e_1 = \sum_{j=1}^{k_0}\sum_{i=0}^{n_j^{(0)}-1} \frac{-\omega_{j,i}}{(z-\tau_j)^{i+1}} + \sum_{j=1}^{k_1}\sum_{i=0}^{n_j^{(1)}-1} z^{i} \sigma_{j,i}
\end{equation*}
where
\begin{align*}
    \omega_{j,i} &= (v^{(j)})^\T \big(S_{n_j^{(0)}}\big)^i u^{(j)},\\
    \sigma_{j,i} &= (v^{(k_0+j)})^\T \big(S_{n_j^{(1)}}\big)^i u^{(k_0+j)}.
\end{align*}
\end{theorem}
\begin{proof} We have
\begin{equation} \label{eq:Tninv11}
    e_1^\T T_n(z)^{-1}e_1 = \underbrace{e_1^\T V}_{=v^\T} \left((J^{(0)} - zI)^{-1} \oplus (-zJ^{(1)} + I)^{-1} \right) \underbrace{Ue_1}_{=u}.
\end{equation}
Since for any Jordan block $J(\mu,m)$ we have 
\[
J(\mu,m)^{-1} = \sum_{i=0}^{m-1} \frac{(-1)^i}{\mu^{i+1}} \big(S_{m}\big)^i = \sum_{i=0}^{m-1} \frac{-1}{(-\mu)^{i+1}} \big(S_{m}\big)^i,
\]
this gives 
\begin{equation*}
    (J^{(0)} - zI)^{-1} = \bigoplus_{j=1}^{k_0} \sum_{i=0}^{n_j^{(0)}-1} \frac{-1}{(z-\tau_j)^{i+1}} \big(S_{n_j^{(0)}}\big)^i
\end{equation*}
and similarly
\begin{equation*}
    (-z J^{(1)} + I)^{-1} = -z^{-1}(J^{(1)} - z^{-1}I)^{-1} = \bigoplus_{j=1}^{k_1} \sum_{i=0}^{n_j^{(1)}-1} z^{i} \big(S_{n_j^{(1)}}\big)^i.
\end{equation*}
Inserting the last two equalities into \cref{eq:Tninv11} gives
\begin{align*}
    \transpose{e}_1 T_n(z)^{-1} e_1 =& \sum_{j=1}^{k_0}\sum_{i=0}^{n_j^{(0)}-1} \frac{-1}{(z-\tau_j)^{i+1}} \underbrace{(v^{(j)})^\T \big(S_{n_j^{(0)}}\big)^i u^{(j)}}_{=\omega_{j,i}} \\
    {}& + \sum_{j=1}^{k_1}\sum_{i=0}^{n_j^{(1)}-1} z^{i} \underbrace{(v^{(k_0+j)})^\T \big(S_{n_j^{(1)}}\big)^i u^{(k_0+j)}}_{=\sigma_{j,i}}. \qedhere
\end{align*}
\end{proof}

We note that a similar result was given in \cite{Berljafa} under the additional assumption that $T_n^{(1)}$ is nonsingular if there are higher-order poles.

\begin{remark}
    The coefficients $\omega_{j,i}$ and $\sigma_{j,i}$ are essentially determined by the vectors $u = Ue_1$ and $\transpose{v}= \transpose{e}_1V$. By choosing vectors other than $\transpose{e}_1$ and $e_1$, one would still obtain a rational function with the same poles and same multiplicities. Moreover, the polynomial part would have degree at most $\max_j n_j^{(1)}-1$.
\end{remark}

\subsection{Generalized Sylvester equation}
There is an interesting connection of the linear system $T_n(A)x = e_1 \otimes v$ to Sylvester-type matrix equations which we shortly discuss here. Let us denote by $\vectorize{}$ the operator which maps a matrix to the vector obtained by stacking its columns from left to right. Then,
see for example \cite{vanLoan,vanLoan2}, for compatible matrices $A,X$ and $B$, we have 
\begin{equation*}
	\vectorize{AXB} = (\transpose{B}\otimes A)\vectorize{X}.
\end{equation*}
Because of this, the {\em Sylvester equation}  
\begin{equation*}
	AX + XB = C .
\end{equation*}
is equivalent to the linear system 
\begin{equation*}
	(I \otimes A + \transpose{B} \otimes I)\vectorize{X} = \vectorize{C}.
\end{equation*}
This is the basis for the following corollary.

\begin{corollary} \label{cor:matrix_equation}
	The linear system \cref{eq:5d_linearsystem}	with $\Tfive(A) = \Tfive^{(0)}\otimes I - \Tfive^{(1)}\otimes A$ is equivalent to the matrix equation
	\begin{equation} \label{eq:Sylvester_type}
		X(\Tfive^{(0)})^{\T} - AX(\Tfive^{(1)})^{\T} = v \transpose{e}_1 \quad \mbox{where } x = \vectorize{X}.
	\end{equation}
\end{corollary}

Solution methods for Sylvester equations have been an active research area in recent years. See, e.g., \cite{Simoncini16}, for a review. Consequently, it might be worthwhile to adapt such methods to \cref{eq:Sylvester_type}---an idea that, however, we do not develop any further here.

\subsection{Block Jacobi and block Gauß-Seidel}
If one wants to compute $r(A)v$, one might use its partial fraction expansion
(see \cref{thm:5d_pfe}) to get
\begin{equation}\label{eq:pfe}
    r(A)v = -\sum_{j=1}^{k_0}\sum_{i=0}^{n_j^{(0)}-1} \omega_{j,i}(A-\tau_jI)^{-(i+1)}v + \sum_{j=1}^{k_1}\sum_{i=0}^{n_j^{(1)}-1} \sigma_{j,i} A^{i}v.
\end{equation}
This involves solving the shifted systems
\begin{equation*}
     (A-\tau_jI)^{i+1}x = v.
\end{equation*}

We will now show that for the CF-system~\cref{eq:5d_linearsystem}, block iterative methods like block 
Jacobi and block Gauß-Seidel with a specific choice for the blocks can be expected to expose the 
same convergence properties as their nonblock versions on the shifted systems arising from the partial fraction expansion.

As a preparation, we introduce two theorems about the eigenvalues of pencils. 
We use the standard terminology for eigenvalues of a pencil as, e.g., in \cite{Stewart}: The spectrum $\spec(T-z\tilde{T})$ of a pencil $T-z\tilde{T}$ are all $\lambda \in \complns$ for which there exists a nonzero eigenvector $x$ such that $Tx - \lambda \tilde{T}x = 0$. In addition, 
the pencil has an eigenvalue at infinity if $\tilde{T}$ is singular.
\begin{theorem}\label{thm:ews_pencils_diag}
	Let the Weierstrass canonical form of the regular pencil $T^{(0)} -z T^{(1)} \in \complns^{(n+1)\times(n+1)}$ be diagonal, i.e., $k_0 = n^{(0)}$, $k_1 = n^{(1)}$ and
    \begin{equation*}
        U(T^{(0)} - zT^{(1)})V =
    	\underbrace{\begin{bmatrix}
             D\\
             & I_{n^{(1)}}
        \end{bmatrix}}_{K^{(0)}} - z
     	\underbrace{\begin{bmatrix}
             I_{n^{(0)}} \\
             & 0
        \end{bmatrix}}_{K^{(1)}}
        \enspace \mbox{ with } D = \diag{\tau_1,\ldots,\tau_{k_0}}.
    \end{equation*}
    Furthermore, define the following matrices:
    \begin{align*}
        T &= T^{(0)} \otimes I - T^{(1)} \otimes A, \\
        \widetilde{T} &= T^{(0)} \otimes I - T^{(1)} \otimes \widetilde{A}, \enspace \mbox{ with } A, \widetilde{A} \in \complns^{m\times m}.
    \end{align*}
    Then the following holds:
    \begin{enumerate}[(i)]%[label={(\roman*)}]
        \item $\spec(T-z\tilde{T}) =   \begin{cases}
                                \bigcup_{i=1}^{k_0} \spec\big((\tau_i I - A) - z(\tau_i I - \tilde{A})\big) & \text{if } n^{(1)} = 0,\\
                                \bigcup_{i=1}^{k_0} \spec\big((\tau_i I - A) - z(\tau_i I - \tilde{A})\big) \cup \{1\} & \text{if } n^{(1)} > 0.
                            \end{cases}$
        \item The pencil $T-z\tilde{T}$ has an eigenvalue at infinity if and only if one of the pencils $(\tau_\alpha I - A) - z(\tau_\alpha I - \tilde{A})$ has an eigenvalue at infinity.
        \item If $w$ is an eigenvector of the pencil $(\tau_\alpha I - A) - z(\tau_\alpha I - \tilde{A})$ with eigenvalue $\lambda$, then there exists a vector $v\neq 0$ such that $v\otimes w$ is an eigenvector of $T-z\tilde{T}$ with eigenvalue $\lambda$.
    \end{enumerate}
\end{theorem}
\begin{proof}
    Multiplying the eigen equation $Tx = \lambda \tilde{T}x $ with $U \otimes I$ from the left and putting $y = (V^{-1} \otimes I)x$ gives the equivalent equation
    \begin{equation} \label{eq:eig_for_y}
        \big(K^{(0)} \otimes I - K^{(1)} \otimes A\big) y = \lambda \big(K^{(0)} \otimes I - K^{(1)} \otimes \tilde{A}\big) y .
    \end{equation}
    The matrices $K^{(0)} \otimes I - K^{(1)} \otimes A$ and $K^{(0)} \otimes I - K^{(1)} \otimes \tilde{A}$ are both block diagonal with $n+1$ diagonal blocks of size $m$. We use the index $\alpha \in\{1,\ldots,n+1\}$ to denote such a block.

    Now assume that \cref{eq:eig_for_y} holds for some $y \neq 0$ and let $\alpha$ be a block for which $y_\alpha \neq 0$. Then block $\alpha$ in \cref{eq:eig_for_y} reads 
    \begin{align*}
        (\tau_\alpha I - A) y_\alpha &= \lambda(\tau_\alpha I - \tilde{A})y_\alpha  &&\mbox{if } \alpha \leq n^{(0)}, \\
        Iy_\alpha &= \lambda Iy_\alpha  &&\mbox{if } n^{(0)} < \alpha \leq n+1.
    \end{align*}
    This proves that the set to the left is contained in the set to the right in (i).
    For the opposite inclusion, assume that for some $\alpha \in 
    \{1, \ldots, n^{(0)}\}$ we have $(\tau_\alpha I - A) w = \lambda(\tau_\alpha I - \tilde{A})w$ for some vector $w \neq 0$. Then the vector $y$ with block components
    \[
    y_\beta = \left\{ \begin{array}{rl} w & \mbox{if } \beta = \alpha, \\
                                0 &  \mbox{otherwise} \end{array} \right .
    \]
    is nonzero and satisfies \cref{eq:eig_for_y}. This, in passing, proves (iii) because $y=e_\alpha \otimes w$ with $e_\alpha$ being a unit vector gives $x=(Ve_\alpha)\otimes w = v\otimes w$ as an eigenvector of $T-z\tilde{T}$. In case $n^{(0)} < n+1$, for $\alpha \in \{n^{(0)}+1,\ldots,n+1\}$ the same construction with $w$ being just any nonzero vector from $\complns^m$ gives an eigenvector $y$ with eigenvalue 1 in \cref{eq:eig_for_y} which concludes the proof of (i).
    
    Finally for (ii), $\infty$ being an eigenvalue of the pencil $T-z\tilde{T}$ is equivalent to $\tilde{T}$ being singular. Then $K^{(0)}\otimes I - K^{(1)} \otimes \tilde{A}$ is singular which means that for some $\alpha \in \{1, \ldots, n^{(0)}\}$ the diagonal block $\tau_\alpha I - \tilde{A}$ is singular. This in turn implies that $\infty$ is an eigenvalue of the pencil $(\tau_\alpha I - A) - z(\tau_\alpha I - \tilde{A})$.
\end{proof}

\begin{remark}\label{re:diagonal_simple_pfe}
    If and only if the pencil $\Tfive^{(0)} - z\Tfive^{(1)}$ of the CF-matrix $\Tfive$ has a diagonal Weierstrass form, the corresponding partial fraction expansion has only simple poles and its polynomial part is constant, since it then simplifies to
    \begin{equation*}
        \transpose{e}_1 T_n(z)^{-1} e_1 = \sum_{j=1}^{k_0} \frac{\omega_{j,0}}{\tau_j-z} + \sum_{j=1}^{k_1} \sigma_{j,0}.
    \end{equation*}
\end{remark}

\begin{theorem}\label{thm:ews_pencil_gen}
	Let the Weierstrass canonical form of the pencil $T^{(0)} - zT^{(1)}$ in \cref{thm:ews_pencils_diag} not be diagonal, i.e.,
	\begin{equation*}
        U(T^{(0)} - zT^{(1)})V =
    	\begin{bmatrix}
             J^{(0)}\\
             & I_{n^{(1)}}
        \end{bmatrix} - z
     	\begin{bmatrix}
             I_{n^{(0)}} \\
             & J^{(1)}
        \end{bmatrix}.
    \end{equation*}
    Then (i) of \cref{thm:ews_pencils_diag} still holds,
    \begin{equation*}
        \spec(T-z\tilde{T}) =   \begin{cases}
                                    \bigcup_{i=1}^{k_0} \spec\big((\tau_i I - A) - z(\tau_i I - \tilde{A})\big) & \text{if } n^{(1)} = 0,\\
                                    \bigcup_{i=1}^{k_0} \spec\big((\tau_i I - A) - z(\tau_i I - \tilde{A})\big) \cup \{1\} & \text{if } n^{(1)} > 0.
                                \end{cases}
    \end{equation*}
\end{theorem}
\begin{proof}
	To circumvent the nontrivial Jordan blocks in $J^{(0)}$, $J^{(1)}$, we perturb the pencil $T^{(0)} - zT^{(1)}$ using the pencil $E - zF$ defined such that
	\begin{alignat*}{2}
		UEV &= \begin{bmatrix}
             \varepsilon E^{(0)}\\
             & 0
         \end{bmatrix}, \quad &E^{(0)} = \delta \cdot &\diag{2^{-1}, \dots, 2^{-n^{(0)}}},\\
        UFV &= \begin{bmatrix}
             0\\
             & \varepsilon F^{(1)}
         \end{bmatrix}, \quad &F^{(1)} = &\diag{2^{-1}, \dots, 2^{-n^{(1)}}}.
	\end{alignat*}
	We now choose $\delta$ such that for $ 0 < \varepsilon \leq 1$
	the diagonal entries of $J^{(0)} + \varepsilon E^{(0)}$ are pairwise distinct. One such choice is
	\begin{equation*}
		\delta = \begin{cases}
					1 & \text{if } k_0 = 1,\\
					\min_{i\neq j} \abs{\tau_i-\tau_j} & \text{if } k_0 > 1.
				 \end{cases}
	\end{equation*}
	As a consequence, the perturbed pencil $(T^{(0)}+E) - z(T^{(1)}+F)$ is diagonalizable, i.e., there exist matrices $M^{(0)}$ and $M^{(1)}$ such that
	\begin{align*}
		M^{(0)} (J^{(0)} + \varepsilon E^{(0)}) (M^{(0)})^{-1} &= \diag{J^{(0)}_{1,1} + \varepsilon E^{(0)}_{1,1}, \dots, J^{(0)}_{n^{(0)},n^{(0)}} + \varepsilon E^{(0)}_{n^{(0)},n^{(0)}}}, \\
		M^{(1)}(J^{(1)} + \varepsilon F^{(1)}) (M^{(1)})^{-1} &= \diag{\varepsilon F^{(1)}_{1,1}, \dots, \varepsilon F^{(1)}_{n^{(1)},n^{(1)}}}.
	\end{align*}

	We now consider the perturbation $\Tpert - z\widetilde{\Tpert}$ of the pencil $T - z\widetilde{T}$ given by
	\begin{gather*}
		\Tpert = T + (E\otimes I - F\otimes A) = (T^{(0)}+E) \otimes I - (T^{(1)}+F) \otimes A, \\
		\widetilde{\Tpert} = \widetilde{T} + (E\otimes I - F\otimes \widetilde{A}) = (T^{(0)}+E) \otimes I - (T^{(1)}+F) \otimes \widetilde{A}.
	\end{gather*}
	Clearly
	\begin{align*}
		\lim_{\varepsilon \to 0} ((T^{(0)}+E) - z(T^{(1)}+F)) &= T^{(0)} - zT^{(1)}, \\
		\lim_{\varepsilon \to 0} (\Tpert - z\widetilde{\Tpert}) &= T - z\widetilde{T}.
	\end{align*}
	Let $\mu_j$ be the eigenvalues of $\Tpert - z\widetilde{\Tpert}$. Then, see \cite[Theorem~2.1]{Stewart} for example, the eigenvalues $\lambda_j$ of $T - z\widetilde{T}$ can be ordered such that
	\begin{equation*}
		\lim_{\varepsilon \to 0} \mu_j = \lambda_j.
	\end{equation*}
	By construction, we can block-diagonalize the pencil $\Tpert - z\widetilde{\Tpert}$ by multiplying with $\left(\begin{bmatrix} M^{(0)} & 0\\0 & M^{(1)} \end{bmatrix} U\right)\otimes I$ from the left and with $\left(V\begin{bmatrix} (M^{(0)})^{-1} & 0\\0 & (M^{(1)})^{-1} \end{bmatrix}\right)\otimes I$ from the right.
	When applied to the eigen equation, we get
    \begin{align*}
        ((J^{(0)}_{\alpha,\alpha} + \varepsilon E^{(0)}_{\alpha,\alpha}) I - A) y_\alpha &= \mu((J^{(0)}_{\alpha,\alpha} + \varepsilon E^{(0)}_{\alpha,\alpha}) I - \tilde{A})y_\alpha  &&\mbox{if } \alpha \leq n^{(0)}, \\
        (I-\varepsilon F^{(1)}_{\alpha,\alpha} A) y_\alpha &= \mu (I-\varepsilon F^{(1)}_{\alpha,\alpha} \widetilde{A}) y_\alpha  &&\mbox{if } \alpha > n^{(0)}.
    \end{align*}
    By analogy with \cref{thm:ews_pencils_diag}, we have
    \begin{align*}
        \spec(\Tpert-z\tilde{\Tpert}) =& 
        \bigcup_{i=1}^{n^{(0)}} \spec\left(((J^{(0)}_{i,i}+\varepsilon E_{i,i}^{(0)}) I - A) - z((J^{(0)}_{i,i}+\varepsilon E_{i,i}^{(0)}) I - \tilde{A})\right) \\
        &\cup
        \bigcup_{i=1}^{n^{(1)}} \spec\left((I - \varepsilon F_{i,i}^{(1)} A) - z(I - \varepsilon F_{i,i}^{(1)}\tilde{A})\right).
    \end{align*}
	In the limit $\varepsilon \to 0$, we retrieve the eigenvalues of the pencils $(J^{(0)}_{\alpha,\alpha} I - A) - z(J^{(0)}_{\alpha,\alpha} I - \widetilde{A})$ for $\alpha \leq n^{(0)}$ and the eigenvalue 1 for $\alpha > n^{(0)}$. Note that there are only $k_0 < n^{(0)}$ different values $\tau_j$ for $J^{(0)}_{\alpha,\alpha}$ because the Weierstrass form is not diagonal. Thus, we obtain the pencils $(\tau_j I - A) - z(\tau_j I - \widetilde{A})$ with $j \leq k_0$.
\end{proof}

Let us now use the above results for an analysis of splitting-based iterative methods for the system \cref{eq:5d_linearsystem}.
We recall that splitting methods for the system $Ax=b$ can be written as
\begin{equation}
\label{eq:smoother}
	x_{i+1} = x_i + \widetilde{A}^{-1} (b-Ax_i)
\end{equation}
where $\widetilde{A}$ is an easily invertible approximation to $A$. For example, $\widetilde{A}$ may be the diagonal of $A$---giving the Jacobi iteration---or the lower triangular part---giving the Gauß-Seidel iteration.
In a splitting method, the errors  $\epsilon_i = x-x_i$ satisfy 
\begin{equation*}
	\epsilon_{i+1} = (I - \widetilde{A}^{-1}A) \epsilon_i %= \widetilde{A}^{-1}(\widetilde{A}-A) \epsilon_i
\end{equation*}
with the error propagator $M = I - \widetilde{A}^{-1}A$.

We now consider a special block version of the general splitting approach for the system $\Tfive(A) x = e_1\otimes v$ where
$\widetilde{T_n(A)}$ has the form
\[
\widetilde{T_n(A)} = T_n(\widetilde{A}) = \Tfive^{(0)}\otimes I - \Tfive^{(1)}\otimes \widetilde{A}.
\]
Depending on the choice for $\widetilde{A}$, the matrix $T_n(\widetilde{A})$ can be interpreted as resulting from a {\em block} splitting of $T_n(A)$.  To see this, we first observe that for a Kronecker product of two matrices $B\in \complns^{k \times k}$ and $C\in \complns^{\ell \times \ell}$, there exists a permutation matrix $P$ such that
%if $P$ is the perfect shuffle permutation matrix $P_{k,l}$, then
%permutation matrix belonging to the permutation $\pi: \alpha \ell + \beta \to \beta k + \alpha, \alpha=0,\ldots,k-1, \beta = 0,\ldots,\ell-1 $, we have 
\begin{equation*}
    P(B\otimes C)\transpose{P} = C \otimes B.
\end{equation*}
For more details, see, e.g., \cite[eq.~(1.3.5)]{vanLoan}. Thus, %for some permutation matrix $P$ we have
we have
\begin{equation*}
    P\widetilde{T_n(A)} \transpose{P} = I \otimes \Tfive^{(0)} - \widetilde{A} \otimes \Tfive^{(1)},
\end{equation*}
in which the nonzero blocks, which themselves are at most tridiagonal, are determined by the sparsity of $\widetilde{A}$ (and $I$). Therefore, any $\widetilde{A}$ resulting from a splitting for $A$ induces a block splitting $ PT_n(\widetilde{A}) \transpose{P}$ for  $P T_n(A) \transpose{P}$. 
%structure of $\widetilde{A}$ while the blocks themselves are always at most tridiagonal.\footnote{If diagonal entries of $\widetilde{A}$ vanish, the identity matrix in $I \otimes \Tfive^{(0)}$ leads to additional diagonal blocks.} 
For instance, if we take $\widetilde{A}$ to be the diagonal of $A$, the matrix $PT_n(\widetilde{A}) \transpose{P}$ is block diagonal and it represents the approximation to $P T_n(A) \transpose{P}$ corresponding to the block Jacobi splitting. Similarly, if $\widetilde{A}$ is the lower triangular part of $A$, we retrieve the block Gauß-Seidel splitting for the (permuted) matrix $T_n(A)$.

In the next theorem, we relate the spectrum of the error propagator $M_{T_n} = I - T_n(\widetilde{A})^{-1}T_n(A)$ for the CF-matrix to the spectra of error propagators $I-(\widetilde{A}-\tau I)^{-1}(A-\tau I)$ where the shifts $\tau$ are the poles of the rational function represented by the continued fraction.

\begin{theorem}\label{thm:smoother_gen} 
    Given $\widetilde{A}$ as an approximation for $A$, assume that for any $\tau \in \complns$ we take
    \[
    \widetilde{A-\tau I} = \widetilde{A}-\tau I.
    \]
    Let $M(\tau) = I - (\widetilde{A-\tau I})^{-1} (A-\tau I)$ and $M_{T_n} = I - \Tfive(\widetilde{A})^{-1}\Tfive(A) $ be the error propagators for the shifted matrices $A-\tau I$ and the CF-matrix $\Tfive(A)$, respectively.

	Then 
	\begin{equation*}
        \spec(M_{\Tfive}) = \begin{cases}
                                \bigcup_{i=1}^{k_0} \spec(M(\tau_i)) & \text{if } n^{(1)} = 0,\\
                                \bigcup_{i=1}^{k_0} \spec(M(\tau_i)) \cup \{0\} & \text{if } n^{(1)} > 0,
                            \end{cases}
	\end{equation*}
	where $\tau_1,\ldots,\tau_{k_0}$ are the eigenvalues of $\Tfive(A)$ from the Weierstrass canonical form; see \cref{eq:Kronecker_form}. 
	%the set of all eigenvalues of the error propagator $M_{\Tfive}$ is the union of $\{0\}$ and the sets of all eigenvalues of $M(\tau_i)$ where $\tau_i$ are the eigenvalues of the pencil $\Tfive^{(0)} + z\Tfive^{(1)}$.
	%and the generalized eigenvectors of $\Tfive^{(0)}$ with $\Tfive^{(1)}$
	% ,i.e., for
	% \begin{equation*}
	% 	M_{\Tfive} z = \lambda z, \quad \Tfive^{(0)} x = \tau_i \Tfive^{(1)} x
	% \end{equation*}
	% it holds that
	% \begin{equation*}
	% 	M(\tau_i) y = \lambda y.
	% \end{equation*}
\end{theorem}
\begin{proof}
	The eigenvalues of
	\begin{equation*}
		M_{\Tfive} = I - \Tfive(\widetilde{A})^{-1}\Tfive(A)
	\end{equation*}
	are obviously $1-\mu$ where $\mu$ are the eigenvalues of $\Tfive(\widetilde{A})^{-1}\Tfive(A)$. They can be interpreted as the eigenvalues of the pencil $\Tfive(A) - z\Tfive(\widetilde{A})$. \Cref{thm:ews_pencil_gen} tells us that if $n^{(1)} = 0$ they are given by the eigenvalues of the pencils $(\tau_i I - A) - z(\tau_i I - \widetilde{A})$.
	Similarly, the eigenvalues of
	\begin{equation*}
		M(\tau_i) = I - (\widetilde{A}-\tau_iI)^{-1} (A-\tau_iI)
	\end{equation*}
	are given by $1-\mu$ where $\mu$ are the eigenvalues of the pencil $(A-\tau_i I) - z(\widetilde{A}-\tau_i I)$ which are those of the pencil $(\tau_i I - A) - z(\tau_i I - \widetilde{A})$.
	In the case $n^{(1)} > 0$, the same argument as above can be made but in addition $\mu=1$ is possible.
\end{proof}

\Cref{thm:smoother_gen} tells us that the convergence speed for the block splitting method for $\Tfive(A)$, measured by the largest eigenvalue of the error propagator, is the slowest of all convergence speeds for the splitting based methods $(A-\tau_iI)$.

If the partial fraction expansion of the rational function has only simple poles and the polynomial part is constant, we can relate the eigenvectors of the error propagators according to the following corollary.

\begin{corollary}\label{cor:smoother_diag}
	Let the pencil $\Tfive^{(0)} - z\Tfive^{(1)}$ have a diagonal Weierstrass canonical form. Then in addition to \cref{thm:smoother_gen}, %for every eigenvalue $\lambda\neq 0$ we can construct a corresponding eigenvector $x$ as a Kronecker product, i.e.,
	\begin{equation*}
		\quad M(\tau_i)w = \lambda w \quad\implies\quad \exists v\neq 0:  M_{\Tfive} (v\otimes w) = \lambda (v\otimes w).
	\end{equation*}
\end{corollary}
\begin{proof}
	By analogy with the proof for \cref{thm:smoother_gen}, the identities in the error propagators do not influence the eigenvectors. What remains are the eigenvectors of the pencils $(A-\tau_i I) - z(\widetilde{A}-\tau_i I)$ which are also the eigenvectors of the pencils $(\tau_i I - A) - z(\tau_i I - \widetilde{A})$ on the left side and the eigenvectors of the pencil $\Tfive(A) - z\Tfive(\widetilde{A})$ on the right side. Part (iii) in \cref{thm:ews_pencils_diag} proves the rest.
\end{proof}

We conclude this section by remarking that each diagonal block of the (permuted) CF-matrix is tridiagonal. So their LU-factorizations can be obtained at low cost implying that the cost for performing one block Jacobi or block Gauss-Seidel iteration for the CF-matrix becomes comparable to the accumulated cost for (nonblock) Jacobi or Gauss-Seidel for all shifted systems $A-\tau I$. Our analysis shows that when working with these block iterations ``stand-alone'', we cannot expect to top the approach where we perform the corresponding nonblock iteration on all shifted systems from the partial fraction. We anticipate, however, that our analysis might be helpful when devising a smoother for a multigrid method (see, e.g., \cite{Trottenberg}) for the CF system in future work.

\section{Numerical examples}\label{sec:numerics}
To illustrate the potential of the CF-matrix approach, we now present numerical examples. Emphasis is placed on how the preconditioned CF-matrix compares with the approach of the partial fraction expansion (see \cref{eq:pfe}). All calculations were done in MATLAB R2021a \cite{matlab}.
\begin{example}[Exponential function]
\label{ex:exp}
	First, consider the exponential function
	\begin{equation*}
		f(A)v = \exp(-A)v.
	\end{equation*}
	According to \cite[eq.~(11.1.3)]{Cuyt}, a regular C-fraction for the inverse of the function, $f(z)^{-1} = \exp(z)$, is given by
	\begin{equation*}
		b_0 = 1,\quad c_1 = 1,\quad c_i = \begin{cases}
											\frac{-1}{2(i-1)} & \text{if } i>1 \text{ is even,}\\
											\frac{1}{2i}    & \text{if } i>1 \text{ is odd.}
											\end{cases}
	\end{equation*}
	We construct the CF-matrix from the contracted continued fraction as illustrated in \cref{ex:contracted} for $n=20$ and apply (full) GMRES \cite{Saad_GMRES} to the resulting system $T_{20}(A)x = e_1\otimes v$ without preconditioning and with ILU(0) \cite{vanderVorst77,Saad} preconditioning.
	For comparison, we also calculate the poles $\tau_j$ of the partial fraction expansion by solving the eigenvalue problem of the pencil $T_{20}^{(0)} - z T_{20}^{(1)}$.\footnote{Note that the Padé approximations in this case have simple poles only. Thus, the Weierstrass canonical form is diagonal.} and apply GMRES to the resulting systems $(A-\tau_j I) x_j = v$. We report, for each iteration, the largest relative residual across all shifts, denoted as PFE in our figures.
	
    For $A$, we use two matrices. First, we take the discrete 2D Laplace operator on a square grid with Dirichlet boundary conditions, i.e.,
	\begin{equation*}
	    A = A_0 \otimes I + I \otimes A_0 \enspace \text{ with }	A_0 = \begin{bmatrix}
				2 & -1\\
				-1 & 2 & \ddots\\
				& \ddots & \ddots & -1 \\
				& & -1 & 2
			\end{bmatrix} \in \complns^{100\times 100},
	\end{equation*}
    which results in $m=100^2$. As a less conventional matrix, we second consider a random sparse nonsingular M-matrix by creating a Z-matrix $B$ via \texttt{sprand} in MATLAB and shifting it by its spectral radius plus $0.1$, i.e.\ $A = B+ (\rho(B)+0.01)I$. 
    To make it comparable to the 2D Laplace example, we use the same size $m=100^2$ and a similar density of $5 \cdot 10^{-4}$.
    The vector $v$ is chosen as a random vector via the function \texttt{rand}. The convergence behavior for both matrices is plotted in \cref{fig:exp}.
    
    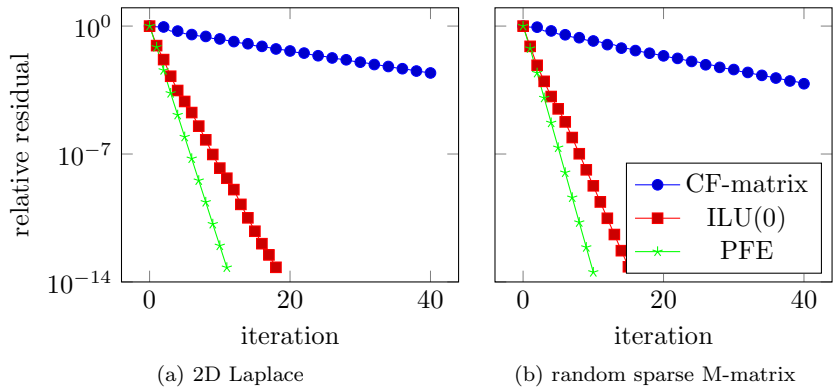
\begin{figure}
        \centering
        \subfloat[2D Laplace]{
        	\begin{tikzpicture}
			\begin{semilogyaxis}[
			    width=0.5\textwidth,
			    xlabel={iteration},
			    ylabel={relative residual},
			    legend pos=south east,
			    ymin=1e-14, ymax=10,
			    ]

				\addplot+[mark repeat=2,] table {plots/exp/laplace_wo.txt};
			    \addlegendentry{CF-matrix}

				\addplot table {plots/exp/laplace_ILU.txt};
			    \addlegendentry{ILU(0)}

                \pgfplotsset{cycle list shift=1}
				\addplot+[green] table {plots/exp/laplace_PFE.txt};
			    \addlegendentry{PFE}
		    
		        \legend{}
			\end{semilogyaxis}
			\end{tikzpicture}
        }
        \subfloat[random sparse M-matrix]{
	    	\begin{tikzpicture}
			\begin{semilogyaxis}[
			    width=0.5\textwidth,
			    xlabel={iteration},
			    yticklabels={,,},
			    legend pos=south east,
			    ymin=1e-14, ymax=10,
			    ]

				\addplot+[mark repeat=2,] table {plots/exp/sprand_wo.txt};
			    \addlegendentry{CF-matrix}

				\addplot table {plots/exp/sprand_ILU.txt};
			    \addlegendentry{ILU(0)}

                \pgfplotsset{cycle list shift=1}
				\addplot+[green] table {plots/exp/sprand_PFE.txt};
			    \addlegendentry{PFE}
		    
			\end{semilogyaxis}
			\end{tikzpicture}
        }
        \caption{Convergence for $\exp(-A)$ via GMRES}
        \label{fig:exp}
	\end{figure}

	Note that the poles $\tau_j$ in this example are complex numbers which necessitates the use of complex arithmetic for the shifted systems $(A-\tau_j I) x_j = v$ even though $A$ and $v$ are real. Since, on the other side, the coefficients of the continued fraction are real, complex arithmetic is not required when working with the CF-matrix. In this sense, we can say that the roughly 50\%--80\% increase in iterations for the ILU(0) preconditioned CF-matrix is approximately compensated for by the fact that we avoid complex arithmetic. 
\end{example}
\begin{example}[Inverse square root]
	We now consider
	\begin{equation*}
		f(A)v = A^{-1/2} v.
	\end{equation*}
	The inverse function $f(z)^{-1} = \sqrt{z}$ is not differentiable at 0 which is why we use the Padé approximations for the function $\hat{f}(z) = \sqrt{z+1}$. A C-fraction for $\hat{f}(z)$ under the condition that $\abs{\operatorname{Arg}(z+1)} < \pi$ is given by \cite[eq.~(11.7.1)]{Cuyt}
	\begin{equation*}
		b_0 = 1,\quad c_1 = \frac{1}{2},\quad c_i = \frac{1}{4} \mbox{ for } i > 1.
	\end{equation*}
	To obtain a pencil for $f(z)$, we exploit $\hat{f}(z-1)=f(z)$ which leads to
	\begin{equation*}
		T_{n}^{(0)} - (z-1)T_{n}^{(1)} = (T_n^{(0)} + T_n^{(1)}) - zT_n^{(1)}.
	\end{equation*}
	Thus, we consider the modified CF-matrix
	\begin{equation*}
		T_{20}(A) = (T_{20}^{(0)} + T_{20}^{(1)})\otimes I - T_{20}^{(1)} \otimes A.
	\end{equation*}
	We repeat the procedure of the previous example\footnote{The Padé approximations for this example have again only simple poles.} and plot the results in \cref{fig:invsqrt}. For the 2D Laplace matrix, we see that the ILU(0) preconditioned CF-matrix approach results in much faster convergence than when using the partial fraction expansion, and for the random spare matrix example, the situation is similar, although significantly less pronounced.

    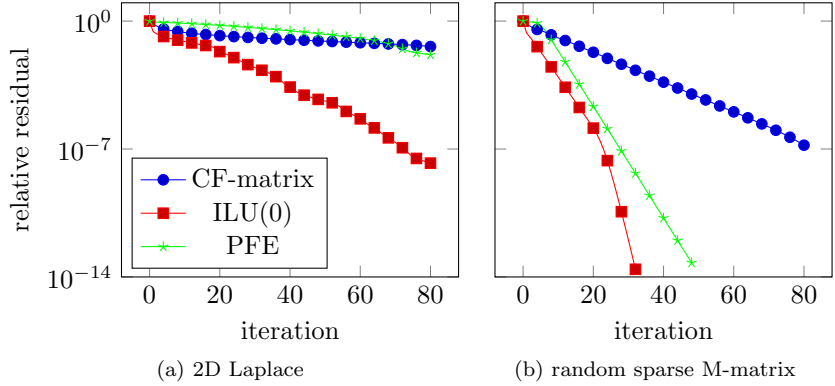
\begin{figure}
        \centering
        \subfloat[2D Laplace]{
        	\begin{tikzpicture}
			\begin{semilogyaxis}[
			    width=0.5\textwidth,
			    xlabel={iteration},
			    ylabel={relative residual},
			    legend pos=south west,
			    ymin=1e-14, ymax=10,
			    mark repeat=4,
			    ]

				\addplot table {plots/invsqrt/laplace_wo.txt};
			    \addlegendentry{CF-matrix}

				\addplot table {plots/invsqrt/laplace_ILU.txt};
			    \addlegendentry{ILU(0)}

                \pgfplotsset{cycle list shift=1}
				\addplot+[green] table {plots/invsqrt/laplace_PFE.txt};
			    \addlegendentry{PFE}
		    
			\end{semilogyaxis}
			\end{tikzpicture}
        }
        \subfloat[random sparse M-matrix]{
	    	\begin{tikzpicture}
			\begin{semilogyaxis}[
			    width=0.5\textwidth,
			    xlabel={iteration},
			    yticklabels={,,},
			    legend pos=south east,	
			    ymin=1e-14, ymax=10,
			    mark repeat=4,
			    ]

				\addplot table {plots/invsqrt/sprand_wo.txt};
			    \addlegendentry{CF-matrix}

				\addplot table {plots/invsqrt/sprand_ILU.txt};
			    \addlegendentry{ILU(0)}

                \pgfplotsset{cycle list shift=1}
				\addplot+[green] table {plots/invsqrt/sprand_PFE.txt};
			    \addlegendentry{PFE}
		    
		        \legend{}
			\end{semilogyaxis}
			\end{tikzpicture}
        }
        \caption{Convergence for $A^{-1/2}$ via GMRES}
        \label{fig:invsqrt}
	\end{figure}
	
	Until now, we only compared the relative residuals of the linear system corresponding to the Pad\'e approximation to the inverse square root. Due
	\begin{equation*}
	    (A^2)^{-1/2} = A^{-1},
	\end{equation*}
	we can determine the {\em error} for the approximation of the inverse square root if we start with a matrix $A$ but approximate $(A^2)^{-1/2}v$ which we compare with the solution of the linear system $Ax = v$.
	We use the 2D Laplace operator again. To prevent the condition number of $A^2$ from being too large, we add a shift of $0.1$ to the matrix $A$, $A \to A+0.01I$. The result is shown in \cref{fig:invsqrt_error}. The plateau to be observed for $n=20$ at $10^{-4}$ for ILU(0)-preconditioned GMRES for the CF-matrix reflects the accuracy of the Pad\`e approximation of  degree $n=20$ to the exact inverse square root.
	 When increasing the degree of the diagonal Padé approximation $n$, we expect this accuracy to increase, as well. Indeed, for $n=40$, the plateau starts to form  at around $10^{-7}$ only (right part of \cref{fig:invsqrt_error}).
	
    \begin{figure}
        \centering
        \subfloat[$n=20$]{
        	\begin{tikzpicture}
			\begin{semilogyaxis}[
			    width=0.5\textwidth,
			    xlabel={iteration},
			    ylabel={relative error},
			    legend pos=south west,
			    ymin=1e-8, ymax=10,
			    mark repeat=2,
			    ]

				\addplot table {plots/invsqrt/error_20_wo.txt};
			    \addlegendentry{CF-matrix}

				\addplot table {plots/invsqrt/error_20_ILU.txt};
			    \addlegendentry{ILU(0)}

                \pgfplotsset{cycle list shift=1}
				\addplot+[green] table {plots/invsqrt/error_20_PFE.txt};
			    \addlegendentry{PFE}
		    
			\end{semilogyaxis}
			\end{tikzpicture}
        }
        \subfloat[$n=40$]{
	    	\begin{tikzpicture}
			\begin{semilogyaxis}[
			    width=0.5\textwidth,
			    xlabel={iteration},
			    yticklabels={,,},
			    legend pos=south east,	
			    ymin=1e-8, ymax=10,
			    mark repeat=2,
			    ]

				\addplot table {plots/invsqrt/error_40_wo.txt};
			    \addlegendentry{CF-matrix}

				\addplot table {plots/invsqrt/error_40_ILU.txt};
			    \addlegendentry{ILU(0)}

                \pgfplotsset{cycle list shift=1}
				\addplot+[green] table {plots/invsqrt/error_40_PFE.txt};
			    \addlegendentry{PFE}
		    
                \legend{}
			\end{semilogyaxis}
			\end{tikzpicture}
        }
        \caption{Comparison of the errors for $(A^2)^{-1/2}$ via GMRES}
        \label{fig:invsqrt_error}
	\end{figure}
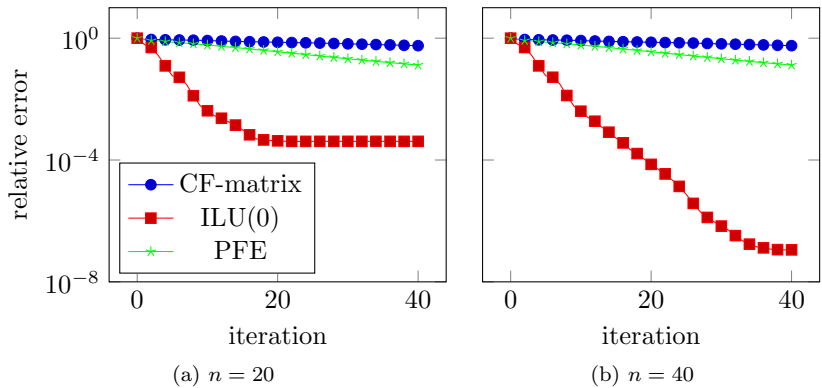
\end{example}

\section{Conclusion and Outlook}
Starting from the observation that the approximants of a continued fraction can be expressed as the (1,1) element of the inverse of a corresponding tridiagonal matrix, we showed that for a rational function $r$ we can describe $r(A)v$ as the solution of a linear system whose coefficient matrix, the CF-matrix, is block tridiagonal. 
What we need is a continued fraction describing $r(z)$. For some continued fractions like regular C-fractions, the resulting CF-matrix contains no higher powers of $A$ and is thus readily available.

For C-fractions, their contractions and similar continued fractions, %the CF-matrix results from a Kronecker structure involving a matrix pencil with tridiagonal matrices.
the Kronecker structure of the CF-matrix involves a matrix pencil with tridiagonal matrices. We showed how the partial fraction expansion of $r(z)$ is related to the Weierstrass canonical form of this pencil. We established a connection to Sylvester-type matrix equations and proved that if $r(z)$ has simple poles only and a certain block structure is used, splitting methods like block Jacobi and block Gauß-Seidel on the CF-matrix exhibit the same convergence properties as their nonblock version on the shifted systems $(A-\tau_i I)$ of the partial fraction expansion. In our numerical experiments, we 
showed results that used (preconditioned) GMRES rather than (block) Jacobi or Gauß-Seidel and obtained fast convergence using ILU(0) as a preconditioner.

An immediate application of the CF-matrix approach is in situations where a continued fraction is available but the computation of its partial fraction expansion is ill-conditioned and thus error-prone. In addition, we expect that the investigations presented here lay the ground for the development of further approaches. Our analysis of the connection between Jacobi and Gauß-Seidel on the CF-matrix with these methods on the shifted matrices can be regarded as the first step towards an understanding of smoothers to be used in a multigrid approach on the CF-matrix. The question of how to develop appropriate coarsening strategies is open and has not been addressed in this paper.

%\section*{Acknowledgement} We would like to thank Stefan Güttel from the University of Manchester for stimulating and enlightening discussions. 

\bibliography{bib}

\end{document}